\documentclass[twoside,a4paper]{article}
\usepackage[english]{babel}
\usepackage[latin1]{inputenc}
\usepackage{amsmath,amssymb,amsthm}
\usepackage{amsfonts}
\usepackage{latexsym}
\usepackage{graphics, graphicx}
\usepackage{upref}
\usepackage{psfrag,rotating}
\usepackage{t1enc}
\usepackage[]{color}
\usepackage{subfig}
\usepackage[T1]{fontenc} 
\usepackage{lmodern}  
\usepackage{comment}     
\usepackage[numbers]{natbib}     
\usepackage{multirow}
\usepackage{caption}
\usepackage{float}
\usepackage[arrow, matrix, curve]{xy}
\usepackage{tikz}
\usepackage{pgfplots}
\usepackage{multicol}
\usepackage[section]{placeins}
\pgfplotsset{compat=1.3}
\usetikzlibrary{plotmarks}
\usepackage{enumerate}
\usepackage{todonotes}
\usepackage{paralist}

\allowdisplaybreaks[1]

\usepackage[most]{tcolorbox}
\newcounter{common}[section]

\numberwithin{equation}{section}

\newcommand{\nz}{\mathbb{N}}
\newcommand{\rz}{\mathbb{R}}

\newcommand{\intO}{\int_\Omega}

\newcommand{\dx}{\,\mathrm{d}x}
\newcommand{\dy}{\,\mathrm{d}y}

\newcommand{\RM}[1]{\MakeUppercase{(\romannumeral #1{})}}
\newcommand{\RMn}[1]{\MakeUppercase{\romannumeral #1{}}}

\newcommand{\HO}{H^{1,2}(\Omega)}
\newcommand{\HhO}{\boldsymbol{H}^{1,2}(\Omega)}
\newcommand{\HOtwo}{H^{2,2}(\Omega)}

\newcommand{\HOzero}{H_0^{1,2}(\Omega)}
\newcommand{\HhOzero}{\boldsymbol{H}_0^{1,2}(\Omega)}
\newcommand{\HOtwozero}{H^{2,2}(\Omega)\cap H_0^{1,2}(\Omega)}

\newcommand{\LOtwo}{L^2(\Omega)}
\newcommand{\LOinf}{L^\infty(\Omega)}
\newcommand{\LlOinf}{\boldsymbol{L}^\infty(\Omega)}
\newcommand{\COinfzero}{C_0^\infty(\Omega)}

\newcommand{\Ltwo}{L^2}

\newcommand{\LOthree}{L^3(\Omega)}
\newcommand{\Lltwo}{\boldsymbol{L}^2}
\newcommand{\LlOtwo}{\boldsymbol{L}^2(\Omega)}
\newcommand{\Lsix}{L^6}
\newcommand{\LOsix}{L^6(\Omega)}
\newcommand{\Llsix}{\boldsymbol{L}^6}

\newcommand{\Hone}{H^1}

\newcommand{\Htwo}{H^2}
\newcommand{\HhOtwo}{\boldsymbol{H}^{2,2}(\Omega)}

\newcommand{\Linf}{L^\infty}
\newcommand{\Llinf}{\boldsymbol{L}^\infty}
\newcommand{\Vv}{\mathbf{V}}
\newcommand{\Hh}{\mathbf{H}}

\newcommand{\ceff}{c}
\newcommand{\keff}{k(\phi)}
\newcommand{\keffmep}{m_\varepsilon*E(k(\phi)) }

\newcommand{\muu}{\mu(\phi)}

\newcommand{\Eeu}{\mep*\mathbf{u}}

\newcommand{\strong}{\rightarrow}
\newcommand{\weak}{\rightharpoonup}

\newcommand{\phit}{\tilde{\phi}}
\newcommand{\phizt}{\tilde{\phi}_0}
\newcommand{\phid}{\phi_D}
\newcommand{\phiz}{\phi_0}
\newcommand{\etat}{\tilde{\eta}}
\newcommand{\etae}{\eta^\varepsilon}

\newcommand{\rhoe}{\rho^\varepsilon}
\newcommand{\Tt}{\tilde{T}}
\newcommand{\phie}{\phi^\varepsilon}
\newcommand{\Te}{T^\varepsilon}
\newcommand{\mep}{m_\varepsilon}
\newcommand{\he}{h_\varepsilon}

\newcommand{\uu}{\mathbf{u}}
\newcommand{\uue}{\mathbf{u}^\varepsilon}
\newcommand{\uut}{\mathbf{\tilde{u}}}

\newcommand{\vv}{\mathbf{v}}

\newcommand{\ee}{\mathbf{e}}

\newcommand{\jjp}{\mathbf{j}_p}

\newcommand{\jj}{\mathbf{j}}
\newcommand{\jje}{\mathbf{j^\varepsilon}}

\newcommand{\gb}{\mathbf{g}}
\newcommand{\Ee}{\boldsymbol{E}}

\newcommand{\material}[1]{\frac{\text{D} #1}{\text{D}\,t}}
\newcommand{\Tinf}{T_\infty}

\newcommand{\bit}{\begin{itemize}}
\newcommand{\eit}{\end{itemize}}
\newcommand{\beq}{\begin{equation}}
\newcommand{\eeq}{\end{equation}}

\newcommand{\supp}{\mathrm{supp}\,}

\newcommand*\Div{\nabla\cdot}

\newtheorem{lemma}{Lemma}[section]

\newtheorem{theorem}[lemma]{Theorem}

\newtheorem{definition}[lemma]{Definition}

\newtheorem{remark}[lemma]{Remark}
\newtheorem{example}[lemma]{Example}

\newcommand{\norm}[1]{\left\lVert#1\right\rVert}
\newcommand{\normLO}[1]{\left\lVert#1\right\rVert_{L^2(\Omega)}}
\newcommand{\normL}[1]{\left\lVert#1\right\rVert_{L^2}}

\newcommand{\normLsix}[1]{\left\lVert#1\right\rVert_{L^6}}
\newcommand{\normLOsix}[1]{\left\lVert#1\right\rVert_{L^6(\Omega)}}

\newcommand{\normLOthree}[1]{\left\lVert#1\right\rVert_{L^3(\Omega)}}

\newcommand{\normLOfour}[1]{\left\lVert#1\right\rVert_{L^4(\Omega)}}

\newcommand{\normLOinf}[1]{\left\lVert#1\right\rVert_{L^\infty(\Omega)}}
\newcommand{\normHO}[1]{\left\lVert#1\right\rVert_{H^{1,2}(\Omega)}}
\newcommand{\normHone}[1]{\left\lVert#1\right\rVert_{H^{1,2}}}

\newcommand{\normHOtwo}[1]{\left\lVert#1\right\rVert_{H^{2,2}(\Omega)}}
\newcommand{\normV}[1]{\left\lVert#1\right\rVert_{\boldsymbol{V}}}

\newcommand{\mat}[1]{\frac{\text{D} #1}{\text{D}  t}}

\usepackage{amsmath}
\usepackage{cancel}
\usepackage{mathtools}
\usepackage[margin=1in]{geometry}
\usepackage{courier}
\usepackage[version-1-compatibility]{siunitx}

\usepackage{tikz}
\usetikzlibrary{patterns}
\usetikzlibrary{arrows}
\usetikzlibrary{arrows,backgrounds} 
\tikzset{%
    body/.style={inner sep=0pt,outer sep=0pt,shape=rectangle,draw,thick,pattern=north east lines wide},
    dimen/.style={<->,>=latex,thin,every rectangle node/.style={fill=white,midway,font=\sffamily}},
    symmetry/.style={dashed,thin},
}
\newcommand\getcurrentref[1]{%
 \ifnumequal{\value{#1}}{0}
  {??}
  {\the\value{#1}}%
}

    \numberwithin{common}{section}

\makeatletter
    {\endtcolorbox}
    \numberwithin{common}{section}

\makeatletter

\makeatletter
\renewenvironment{proof}[1][\proofname]{\par
  \pushQED{\qed}%
  \normalfont \topsep6\p@\@plus6\p@\relax
  \trivlist
  \item[\hskip\labelsep
         \underline{\textbf{#1:}}]
}{
  \popQED\endtrivlist\@endpefalse
}
\makeatother

    \numberwithin{common}{section}

    \numberwithin{common}{section}

\newcommand{\pedrop}[1]%
{\par\medskip\noindent\centerline{\fbox{\parbox{.97\textwidth}{\medskip\centerline{\parbox{.95\textwidth}{\color{magenta} #1}}\medskip}}}\medskip\\}
\newcommand{\Cle}{\lesssim}
\usepackage{hyperref}

\begin{document}

\title{Convective transport in nanofluids: \\ 
   the stationary problem}

\author{Eberhard B\"ansch, Sara Faghih--Naini and Pedro Morin}

\date{\today}

\maketitle 

\begin{abstract}
We analyze the existence of solutions to the stationary problem from a mathematical model 
for convective transport in nanofluids including
thermophoretic effects that is very similar to the 
celebrated model of Buongiorno \citep{Buongiorno2006}. 

\end{abstract}

    

\section{Introduction}

Nanofluids are engineered colloids, invented in the 1990s by Choi and Eastmann 
\citep{Choi1995} and consist of a base fluid and a small 
fraction of nanoparticles, usually below five percent. 
Compared to their base fluids, nanofluids show the following behavior (see  \citep{Buongiorno2006,BaenschFaghih2019} for details):
\begin{compactitem} 
\item Much higher and strongly temperature-dependent thermal conductivity,
\item Heat transfer coefficient increase going beyond the mere thermal 
conductivity effect, 
\item Abnormal viscosity increase,
\item High dispersion stability arising from Brownian motion.
\end{compactitem}
Compared to conventional solid-liquid suspensions with larger particles, nanofluids show a higher 
potential for increasing the heat transfer due to the larger surface between the phases. 
This has motivated their use in numerous cooling or heat exchanging devices~\citep{Saidur2011}.

In the last decades, there were several attempts to develop convective transport models for nanofluids (see \citep{BaenschFaghih2019} and the references therein). 
To overcome some shortcomings of the earlier models, J. Buongiorno \citep{Buongiorno2006} introduced an alternative model in 2006, which has become very popular.  In Sec.\ 2 of his work, he used scaling arguments to conclude 
that Brownian diffusion and thermophoresis are the only important slip mechanisms in laminar 
forced convection of nanofluids. He furthermore estimated the relative importance 
of thermophoresis to be up to two orders of magnitude higher than that of 
Brownian diffusion.

The resulting model is a \textit{two-component four-equation non-homogeneous equilibrium model} and consists 
of a system of partial differential equations (PDEs) for the volumetric fraction of the particles, 
the thermal energy and the momentum of the averaged flow field. The volumetric fraction of the liquid 
phase is determined by the one of the particles of course, and it is assumed that there is a 
common temperature because of thermal equilibrium between the particles and the base fluid.

We will now take a closer look at Brownian diffusion and thermophoresis, 
since they are responsible for the diffusive mass flux in the continuity equation for the nanoparticles in Sec.\ \ref{mass_eq}.
\vspace{0.2cm}\\
\textbf{Brownian diffusion}
\vspace{0.2cm}\\
The random motion of nanoparticles within the base fluid that results from frequent collisions between the nanoparticles and the molecules of the base fluid is called Brownian motion \citep{Buongiorno2006}. It
is described by the Brownian diffusion coefficient $D_B$, which is given by the Einstein-Stokes 
equation \citep{Einstein:1905} as $D_B = \frac{k_B T}{3 \pi \mu_f d_p} $ with Boltzmann constant $k_B = 1.38^{ - 23 }\SI {}{\frac{m^2\, kg}{s^2\,K}}$, absolute temperature $T$, base fluid dynamic viscosity $\mu_f$ and particle diameter $d_p$. The  nanoparticle mass flux due to Brownian diffusion can be modeled as
\begin{align}\label{diff_brown}
\boldsymbol{ j_{p,B}}=-\rho_p D_B \nabla \phi \quad \big[\SI {}{\frac{kg}{m^2\,s}}\big]
\end{align}
with particle density $\rho_p$ and volumetric fraction of particles $\phi$.
\vspace{0.2cm}\\
\textbf{Thermophoresis}
\vspace{0.2cm}\\
The diffusion of particles under the effect of a temperature gradient is called thermophoresis. It is the ``particle equivalent'' of the well-known Soret effect \citep{MCNab1973} for gaseous or liquid mixtures. The thermophoretic velocity can be computed as $\boldsymbol{V_T}=-\beta\frac{\mu_f}{\rho_f}\frac{\nabla T}{T}$ with base fluid density $\rho_f$ and proportionality factor $\beta =0.26\frac{k_f}{2k_f+k_p}$ according to \citep{MCNab1973}. $k_f$ and $k_p$ are the thermal conductivity of the base fluid and the particle material, respectively. The negative sign indicates that particles move in direction opposite to the temperature gradient, i.e.\ from hot to cold regions. The  nanoparticle mass flux due to thermophoretic effects can be calculated by
\begin{align}\label{diff_therm}
\boldsymbol{j_{p,T}}= \rho_p\phi \boldsymbol{V_T}= -\rho_p D_T \phi\frac{\nabla T}{T} \quad \big[\SI {}{\frac{kg}{m^2\,s}}\big]
\end{align}
with thermal diffusion coefficient $D_T = \beta \frac{\mu_f}{\rho_f}$ \citep{Buongiorno2006}.

To adhere to the notation in \citep{Baensch2018}, in the following, the mass flux will be defined without the particle density $\rho_p$, consequently, we do not need to divide by $\rho_p$ in the continuity equation for the nanoparticles and the total mass flux $\jjp$ will be in $\big[\SI {}{\frac{m}{s}}\big]$. Furthermore, deviating from the notation in the literature, the diffusion coefficient $D_T$ does not contain the concentration $\phi$ in our work. \\

In recent years, there have been numerous numerical studies on convective heat transfer in nanofluids 
(see for instance the review \citep{Vanaki2016}), many of them are based on the Buongiorno model. By revisiting the derivation of the model, it was established that the original one is not thermodynamically consistent. 
Thus a slightly different, thermodynamically consistent model was presented in \citep{Baensch2018}. For this model, energy estimates and the existence of a weak solution can be shown.

The aim of this article is to study
the existence of \emph{stationary} solutions of this
thermodynamically consistent model. Interestingly, it turns out that the proof of existence of stationary
solutions is technically more demanding than for the instationary case. 
Some of the arguments from our proofs will be useful for obtaining further regularity results, which are in turn necessary for developing error estimates, which will be published in a forthcoming paper.

The rest of this article is organized as follows. 
In Section \ref{S:model} the modified Buongiorno model is presented. The notion of a 
weak solution to the stationary version of the model is given in Section \ref{S:weak}.
Existence of a weak solution is proved in Section \ref{S:existence}.

\section{The model}\label{S:model}
In this section, we present the model derived in \citep{Baensch2018} for convective transport in nanofluids 
and make some comments on heat transfer mechanisms. The model is based on similar assumptions as those in \citep{Buongiorno2006}:
\begin{itemize}
\item \textit{Incompressible flow}.
\item \textit{The particles and the base fluid are locally in thermal equilibrium}. Since the heat transfer timescales for heat conduction within the nanoparticles and within the base fluid in the vicinity of the nanoparticles are much smaller than the time scales calculated for Brownian diffusion and thermophoresis and also smaller than the possibly turbulent eddy time scale, nanoparticles moving in the surrounding fluid achieve thermal 
equilibrium with it very rapidly.
\item \textit{No chemical reactions} due to the chemical inertness of nanoparticles with the base fluid.
\item \textit{Negligible viscous dissipation}, since the applied heat flux is the dominant energy source in the system.
\item \textit{Negligible radiative heat transfer} because of the relatively low temperatures.
\end{itemize}
In contrast to \citep{Buongiorno2006}, we allow for  
Buoyancy effects through an outer force given by Boussinesq approximation.

Based on these assumptions Buongiorno derived a four-equations non-homogeneous equilibrium model from
balance equations for particle and fluid mass, thermal energy and momentum. In \citep{Baensch2018} this model
was modified in order to get a thermodynamically consistent system.
\subsection{Balance of mass} \label{mass_eq}
The evolution of the volumetric fraction of particles $\phi=\phi(t,x)$ is described by $\mat{\phi} = -  \Div \jjp$~(\citep{Bird1960} Eq.\ 18.3-4) where $\mat{}$ denotes the material derivative $\partial_t +\uu \cdot \nabla$, with $\uu$ the volume averaged velocity described below. 
The flux function $\jjp$, according to \citep{Baensch2018}, is given by
$\jjp= - D_B\nabla \phi - \phi (1-\phi)D_T \frac{\nabla T}{T}$.
As in most practical applications the absolute temperature is much greater than the temperature difference, the denominator $T$ was fixed in \citep{Baensch2018} at a given value $\Tinf >0$,
arriving to the following equation for the balance of particle mass:
\begin{align} \label{concentration}
\mat{\phi} =-  \Div \jjp
           = \Div\Big( D_B\nabla \phi + \phi (1-\phi)D_T \frac{\nabla T}{\Tinf}\Big).
\end{align}

\subsection{Balance of thermal energy}\label{balance_thermal_energy}
If $h =\ceff\rho T$ denotes the enthalpy of the mixture with effective specific heat capacity $\ceff$ and actual density $\rho$, the balance of thermal energy \citep{Baensch2018} reads
\begin{align}
\mat{h} + \Div \big((h_p-h_f) \jjp \big) - \keff\nabla T)=0,
\end{align}
with the subscripts $\cdot_f, \cdot_p$ indicating properties of
the base fluid and the particle phase, respectively. 
The actual density is $\rho = \phi\rho_p + (1-\phi)\rho_f$ and
$\ceff$ is given
as the density weighted heat capacity of the single phases~\citep{Buongiorno2006}:
$\ceff = \big(\phi c_p\rho_p + (1-\phi)c_f\rho_f\big)/\rho$,
which is also assumed in \citep{Baensch2018} and hereafter. The enthalpy of the nanoparticles and the base fluid are given by $h_p =\ceff_p\rho_p T$ and $h_f =\ceff_f\rho_f T$ respectively, whereas
$\keff$ is the concentration-dependent thermal conductivity.
Inserting these expressions and using the abbreviation $\alpha_\eta \coloneqq c_p\rho_p - c_f\rho_f$ yields
\[
\material{h} + \Div(T(c_p\rho_p - c_f\rho_f)\jj_p - \keff\nabla T)=
\material{h} + \alpha_\eta\Div (T\jj_p - \keff\nabla T)=0.
\]
The final form, which will be used in the following, is obtained by defining $\eta:=\rho \ceff = c_f \rho_f + \phi \alpha_\eta\eta$:
\begin{align}
\material{(\eta T)} + \alpha_\eta\Div(T\jj_p) -\Div(\keff\nabla T) =
0.\label{heat}
\end{align}

\subsection{Balance of momentum}
According to \citep{Abels2012} and as stated in \citep{Baensch2018} the \textit{volume averaged} velocity $\uu$ 
fulfills the momentum equation
\[
\material{(\rho\uu)} + \alpha_\rho\Div(\uu \otimes \jj_p) -\Div(\muu D(\uu)) +\nabla p = -\beta T \ee_g
\]
with concentration-dependent viscosity $\mu(\phi)$, pressure $p$, strain rate tensor 
$D(\uu) := \nabla \uu + \nabla \uu^T$, buoyancy coefficient $\beta \ge 0$ with $\ee_g$ a unit vector pointing in the direction of the gravity force and the difference of the densities 
$\alpha_\rho \coloneqq \rho_p-\rho_f$. 
Moreover, $\uu$ is solenoidal (\citep[Eq.\ 18.1-9]{Bird1960}), i.e., 
$ \Div \uu =0$.

\subsection{Difference to Buongiorno model}\label{difference}
The energy equation~\eqref{heat} (and similarly the momentum equation)
can equivalently be written in a non-conservative form as
\begin{align}
\ceff\rho\mat{T} - \Div(\keff\nabla T) = \alpha_\eta\big(D_B \nabla T\cdot \nabla\phi + 
D_T \phi(1-\phi) \frac{|\nabla T|^2}{\Tinf}\big).
\label{heat2}
\end{align}
As noted in \citep{Baensch2018}, this equation is of the same form as the heat equation derived in \citep{Buongiorno2006} except
for the factors $(1-\phi)$ in the definition of $\jj_p$
and $\alpha_\eta= c_p\rho_p - c_f\rho_f$  instead of $c_p\rho_p$. In \citep{Buongiorno2006} also the term
$\alpha_\rho\Div(\uu \otimes \jj_p)$ is neglected in the momentum equation.

These small modifications (being of minor significance in most applications) to the
model by Buongiorno make the difference that the model from \citep{Baensch2018} is 
thermodynamically consistent and in turn enjoys provable stability estimates from which
also existence of solutions can be rigorously shown.

\subsection{Stationary problem}
In this article we are interested in stationary solutions to the aforementioned model, i.e., triples $\phi, \, T,\, \uu$ and $p$, such that, in $\Omega$,
\begin{align}
\uu \cdot \nabla \phi + \nabla \cdot \jj_p &=0, \label{stationary_strong_phi}\\
\uu \cdot \nabla (\eta T )+ \nabla \cdot (\alpha_\eta T \jj_p - \keff \nabla T)&=f,\label{stationary_strong_T}\\
\uu \cdot \nabla (\rho \uu) +\alpha_\rho \nabla \cdot (\uu \otimes \jj_p - \muu D(\uu)) +\nabla p + \beta T \ee_g&=\gb,\label{stationary_strong_u}\\
\nabla \cdot \uu &=0 \label{stationary_strong_div_u}
\end{align}
where $\jj_p= - \nabla \phi - h(\phi) \frac{\nabla T}{T_\infty}$ and 
\begin{equation} \label{h_def}
h(z) = z^+ (1-z)^+\qquad \text{with}\quad z^+ = \max\{0,z\}.
\end{equation}
We consider the following boundary conditions:
\begin{align} \label{boundary}
\phi=b , \quad T=0, \quad \uu=\boldsymbol{0} \qquad \text{ on } \partial \Omega,
\end{align}
where we assume $0\leq b \leq 1$. Other boundary conditions could be treated as well. However, for the
ease of presentation we choose purely Dirichlet conditions.
Suitable assumptions on the problem data $f$, $\gb$ and $b$ will be stated in the next section.  
Here, we have added nonzero source terms to the temperature and momentum equations in order to 
get non-trivial solutions.

\section{Weak solutions to the stationary problem}\label{S:weak}

In this section we show existence of weak solutions to the stationary problem. 
Because of the rather strong nonlinearity we first have to regularize the problem by mollification.
We use a fixed-point argument to infer that there exist  solutions $(\phie, \Te, \uue)$ 
to this regularized system (with $\varepsilon$ denoting the regularization parameter), and then show 
that these solutions converge to  a solution $(\phi, T, \uu)$ of the original system as $\varepsilon \strong 0$.

\subsection{Notation and assumptions}\label{S:NotationAssumptions}

We will use the following notation:

\begin{itemize}

\item $\Vv:=\{\vv\in \HhOzero | \Div \vv =0\}$, where $H_0^{1,2}(\Omega)$ is the closure of $C^\infty_0(\Omega)$ (the space of test functions) in $H^{1,2}(\Omega)$, the latter being the space of weakly differentiable functions with square integrable first order derivatives; $\HhOzero = \big(H_0^{1,2}(\Omega)\big)^d$.

\item  $\Hh$ is defined as the closure of $\Vv$ in $\Lltwo$.

\item The Euclidean scalar product of vectors $\boldsymbol{a}$ and $\boldsymbol{b} \in \rz^d$ is $\boldsymbol{a} \cdot \boldsymbol{b}$. The analog for matrices $\boldsymbol{A}, \boldsymbol{B} \in \rz^{d_1\times d_2}$ is $
\boldsymbol{A} : \boldsymbol{B} = \sum_{i=1}^{d_1}\sum_{j=1}^{d_2} \boldsymbol{A}_{i,j}\,\boldsymbol{B}_{i,j}.$

\item The tensor product $\boldsymbol{a}\otimes \boldsymbol{b}\in\rz^{n\times m}$ of 
two vectors $\boldsymbol{a}\in \rz^n$ and $\boldsymbol{b}\in \rz^m$ is defined as 
$$(\boldsymbol{a}\otimes \boldsymbol{b})_{i,j} = a_ib_j,\qquad i=1,\dots n; \quad j=1,\dots m.$$
\end{itemize}

We make the following assumptions:

\begin{itemize}
	\item $\Omega\subseteq \rz^d$ is an open, bounded domain 
	with a globally Lipschitz boundary, and satisfies a uniform exterior sphere condition,
	for $d=2$ or $3$. This is condition is for instance fulfilled for convex domains. Otherwise,
        the ``concave parts'' of the boundary have to be smooth ($C^2$).
	\item $k, \mu : [0,1] \rightarrow \rz$ are Lipschitz continuous and  $\keff \geq k_0 >0$, $\muu \geq \mu_0 >0$.
	\item $f,\gb \in L^2(\Omega)$.
	\item Boundary data $b$ can be extended to a function $\phid\in H^{2,2}(\Omega)$
          with $\Delta \phid=0$.
\end{itemize}

We will need the following regularity result. Although it seems fairly standard, 
we give a proof here for completeness.

\begin{lemma}\label{Lemma:regularity}
Let $6/5 \leq p \le 2$, and
let $f \in L^p(\Omega)$ so that $f\in H^{-1,2}(\Omega)$. 
Furthermore let $a \in C^{0,1}(\Omega)$, with $0< a_0 \le a \le a_\infty$ in $\Omega$. 
Then the unique weak solution $\zeta \in H_0^1(\Omega)$ of 
$$
-\Div(a \nabla \zeta) = f \quad\text{in } \Omega, \quad\text{and } \zeta =0 
 \quad\text{on }\partial\Omega
$$ 
satisfies $\zeta \in H^{2,p}(\Omega)$ with $\| \zeta \|_{H^{2,p}(\Omega)} \le C \| f\|_{L^p(\Omega)}$, and $C = C(a,p,\Omega)$. 
\end{lemma}

\begin{proof}
Clearly, $\zeta$ fulfills the estimate
$$
 \| \zeta\|_{H^{1,2}(\Omega)} \leq C\|f\|_{L^p(\Omega)}.
$$
Also, $\zeta\in H^{1,2}_0(\Omega)$ is the weak solution to
$$
-\Delta \zeta = \underbrace{\frac{1}{a}(f + \nabla a\cdot\nabla\zeta)}_{=:\tilde f}
\quad\text{in } \Omega, \quad\text{and } \zeta =0 
\quad\text{on }\partial\Omega.
$$
Since $\zeta\in H^{1,2}(\Omega)$,  it follows that $\tilde{f}\in L^{\min{\{p,2\}}}(\Omega)$.
The regularity result from Corollary~1 in~\cite{Fromm93} (see also Remark on top of page 232 in~\cite{Fromm93}) readily implies
that $\zeta\in H^{2,p}(\Omega)$ and $\| \zeta \|_{H^{2,p}(\Omega)} \le C \| f\|_{L^p(\Omega)}$.
\end{proof}

Hereafter, $A \Cle B$ denotes $A \le C B$ with a constant $C$ that depends only on the domain $\Omega$ and the fixed parameters of the problem, but not on the problem data or solution.

\subsection{Weak formulation of the stationary problem}

We now assume, without loss of generality, that $T_\infty=1$, $\alpha_\eta=\alpha_\rho = 1$, $\eta = 1+\phi$
and $\rho = 1 + \phi$.
\begin{definition}[Weak solution to the stationary problem]\label{Def_weak}
Given right hand sides $f\in\Ltwo$, $\gb\in\Hh$,
a weak solution to problem \eqref{stationary_strong_phi}, \eqref{stationary_strong_T}, 
\eqref{stationary_strong_u} and \eqref{stationary_strong_div_u} with boundary conditions 
\eqref{boundary} is a triple $(\phi, T, \uu)$ 
with $\phi \in \HO$, $\phi-\phid\in \HOzero$, $T \in \HOzero$, $\uu \in \Vv$  
and  $0\leq \phi \leq 1$ a.e. fulfilling
\begin{align}
\intO \uu\cdot\nabla\phi \psi-\jj_p \cdot \nabla \psi \dx&=0, \label{phi_orig} \\[1ex]
\intO \uu\cdot \nabla(\eta T)\varphi  + \big(-T \jj_p+\keff\nabla T\big) \cdot \nabla \varphi \dx&=
         \intO f \varphi \dx,\label{T_orig}\\[1ex]
\intO \uu \cdot \nabla(\rho\uu) \cdot \vv -\uu \otimes\jj_p:\nabla \vv \dx + \intO \frac{\muu}{2} 
D(\uu) : D(\vv) \dx
+\intO \beta T \ee_g \cdot \vv \dx &=\intO \gb \cdot \vv \dx\label{u_orig}
\end{align}
with
$ \jj_p= - \nabla \phi - h(\phi) \nabla T$,
for all $\psi, \, \varphi \in C^\infty_0(\Omega), \, \vv \in  \boldsymbol{C}^\infty_0(\Omega) \text{ with } \Div \vv =0$.
\end{definition}

The mathematical challenge of the above system lies in the rather strong nonlinearity
given by the terms $\Div (T\jjp)$ and $\Div(\uu\otimes \jjp)$ that give rise
to expressions like $|\nabla T|^2$. Fortunately, the system exhibits an energy
estimate in $H^{1,2}\times H^{1,2}\times \Vv$.
To exploit this structure, we need to regularize the system and then
pass to the limit.
\subsection{Mollification}

As already mentioned, in order to prove existence of solutions
we need to regularize certain terms in the equations.
This will be done by mollification. To this end, in what follows, functions with zero boundary values
on $\partial\Omega$ are extended by zero outside of $\Omega$ without further noting.
\begin{definition}{Mollifier and convolution (\citep[App. C.5]{Evans2010})}
Define $m \in C^\infty(\rz^d)$ by 
$m(x)=C \exp\left(\frac{1}{|x|^2-1}\right)$ for $|x|< 1$, and $m(x) = 0$ otherwise,
with the constant $C$ such that $\int_{\rz^d}m\dx=1$.
For each $\varepsilon>0$ set 
$\mep(x)\coloneqq m(x/\varepsilon)/\varepsilon^n$.
The functions $\mep$ are thus in $C^\infty(\rz^d)$ and $\int_{\rz^d}\mep\dx=1$.
The $\varepsilon$-mollification of a function $f \in L^p(\rz^d)$ is defined as the convolution with this mollifier:
\begin{align*}
f^\varepsilon(x)=(\mep * f)(x) \coloneqq \int_{\rz^n} \mep(x-y)f(y) \dy.
\end{align*}
\end{definition}

\begin{lemma}[Properties of mollifiers and convolutions]

\label{Lemma_convolution_mollifier}
For a convolution of a function $f \in L^p(\rz^n)$ with a mollifier $\mep$ the following properties hold for all $1\leq p < \infty$, \citep[App. C3, Ch.\ 5.3]{Evans2010}:
\begin{enumerate}[(i)]
\item $ \mep * f = f * \mep \in C^\infty(\rz^n)$, and
      $\norm{f-\mep *f}_{L^p(\Omega)} \strong 0 \text{ as } \varepsilon \strong 0$,

\item If $f \in H^{1,p}(\Omega)$: $\norm{\nabla(f-\mep *f)}_{L^p(\Omega)} \strong 0 \text{ as } \varepsilon \strong 0$,

\item $\| \mep * f \|_{L^p(\Omega)}\leq \|f\|_{L^p(\rz^n)} $,

\item the classical $\alpha^{th}$-partial derivative $D^\alpha$ of the smooth function $f_\varepsilon$ 
equals the $\varepsilon$-mollification of the weak $\alpha^{th}$-partial derivative of $f$ and it is also equal to the $\alpha^{th}$-partial derivative of the mollifier convoluted with $f$, i.e.\ 
$ D^\alpha(\mep *f)=(D^\alpha \mep) *f= \mep *D^\alpha f$,

\item\label{Lemma_convolution_mollifier.v} $\norm{\mep*f}_{H^{k,q}(\Omega)}\leq C_\varepsilon(k,q)\norm{f}_{L^p(\rz^n)}$ for $0 \leq k<  \infty$, $1 \leq q \leq \infty$.
\end{enumerate}
\end{lemma}

The following property will be used later in the convergence proof:
\begin{lemma}{}
	If $f,g \in L^1(\Omega)$ and we extend them as zero to $\rz^d$, then
\begin{align}
\intO f(x)  (\mep *g)(x) \dx=\intO (\mep* f)(x) g(x) \dx \qquad\text{for all }\varepsilon > 0.
\end{align} 
\end{lemma}

The proof of this result is straightforward using Fubini's theorem and the fact that the extended functions are zero outside of $\Omega$.

Whenever we apply the mollification to a function in $ f \in L^p(\Omega)$ we understand that it is applied 
to the natural extension by zero outside $\Omega$ so that $\mep* f$ is always defined on all of $\rz^n$. 
When $f \in \HOzero$, its extension by zero belongs to $H^{1,2}(\rz^n)$, but if $f \in \HO \setminus \HOzero$, 
we will apply the mollification to an extension through the following operator.

\begin{lemma}\label{L:extension}
	Under the assumptions made on $\Omega$ there exists a linear and bounded  extension operator 
     $E$ from $H^{1,2}(\Omega)$ to $H^{1,2}(\rz^n)$.
\end{lemma}

\section{Proof of existence}\label{S:existence}

To prove the existence of a weak solution we proceed as follows. We first regularize the problem with a small parameter $\varepsilon > 0$, and prove existence of solution to this problem using Leray-Schauder's fixed point theorem. 
We will then prove estimates independent of $\varepsilon$ as it tends to zero that will lead to convergence 
of a subsequence to a solution to the original problem.

\subsection{Regularized problem}\label{S:RegularizedProblem}

For a fixed $0<\varepsilon \le 1/4$, we define the mapping 
\begin{align}\label{P:regularized}
G: \HOzero \times \HOzero \times \Vv \rightarrow \HOzero \times \HOzero \times \Vv, \quad G(\phiz,T, \uu)=(\phizt,\Tt, \uut)
\end{align}
as follows. Given $\phiz \in  \HOzero$ such that $0 \le \phi := \phi_0+\phi_D \le 1$, $T \in \HOzero$ and $\uu \in \Vv$, let $\phizt  \in  \HOzero$ be such that $\phit=\phizt +\phid  \in \HO$ satisfies, for every $\psi \in \HOzero$,
\begin{align} \label{phi_tilde}
\intO   \Eeu \cdot\nabla \phit \, \psi \dx  
+\intO  \nabla \phit \cdot \nabla \psi \dx 
+\intO \he(\phi)\nabla \mep*T \cdot \nabla \psi \dx=0.
\end{align}
Here, since $\uu \in \Vv \subset H_0^{1,2}(\Omega)^d$ and $T \in H_0^{1,2}(\Omega)$, the convolutions have been performed with their natural extensions as zero outside $\Omega$.
Furthermore,
\begin{align}\he(z)=\begin{cases}
0 & \text{ for } z<0,\\
-\frac{1}{\varepsilon^2}z^3+(\frac{2}{\varepsilon}-1)z^2 & \text{ for } 0\leq z<\varepsilon,\\
h(z) & \text{ for } \varepsilon\leq z\leq 1-\varepsilon,\\
\frac{1}{\varepsilon^2}z^3+(-\frac{3}{\varepsilon^2}+\frac{2}{\varepsilon}-1) z^2+(\frac{3}{\varepsilon^2}-\frac{4}{\varepsilon}+2)z -\frac{1}{\varepsilon^2}+\frac{2}{\varepsilon}-1 & \text{ for } 1-\varepsilon<z\leq 1,\\
0 & \text{ for } 1<z,
\end{cases}\label{hep_def}
\end{align}
is a smooth approximation of $h(z)=z^+(1-z)^+ $
that satisfies $|h(z)-\he(z)|\leq \varepsilon-\varepsilon^2$, 
$\supp h = \supp \he=[0,1]$, 
and $h=\he$ in $[\varepsilon, 1-\varepsilon]$. 
Moreover, $h'(z)=\chi_{(0,1)}(z) (1-2z)$ so that $|h'(z)|\leq 1$, and also $\|\he'\|_{L^\infty}\le 1$ for every $\varepsilon > 0$.

For $\phit$ fulfilling Eq. \eqref{phi_tilde} we define
\begin{align}\label{flux_tilde}
\jjp=-\nabla \phit -\he(\phi)\nabla \mep*T
\quad\text{and}\quad \tilde{\eta} = 1 + \tilde{\phi},
\end{align}
and let $\Tt \in \HOzero$ be such that, for every $\varphi \in \HOzero$,
\begin{align}\label{T_tilde}
\underbrace{\intO \mep*k(\phi) \nabla \Tt\cdot \nabla \varphi \dx 
     +\intO \Eeu \cdot \nabla (\etat \Tt)   \varphi \dx 
- \intO \Tt \jjp \cdot \nabla \varphi \dx}_{=:b_T[\tilde T, \varphi]}=\intO f \varphi \dx.
\end{align}
Finally, we let $\uut \in \Vv$ be such that, for every $\vv \in \Vv$,
\begin{equation}\label{u_tilde}
\begin{split}
\intO \frac{\mep*\muu}{2} 
D(\uut) : D(\vv) \dx  &+ \intO \Eeu \cdot \nabla (\rho\uut) \cdot \vv \dx \\
&- \intO \uut \otimes \jjp: \nabla \vv \dx 
+ \intO \beta T \ee_g \cdot \vv \dx=\intO \gb \cdot \vv \dx.
\end{split}
\end{equation}

\begin{lemma}\label{Lemma_phit}
For any $0<\varepsilon \leq 1/4$ there exists a unique solution $\phit$ to Eq.\ \eqref{phi_tilde} and it satisfies
\begin{align*}
\normL{\nabla \phit}&\le \normLO{\nabla \phizt}+\normLO{\nabla \phid} \Cle  
   \normLO{\nabla T} +\norm{b}_{H^{1/2,2}(\partial \Omega)}(1+ \normLOfour{\uu}).
\end{align*}
Moreover $\phit \in \HOtwo$ and 
\[
\normHOtwo{ \phit}\leq  C_\varepsilon \normLO{T} 
   +C\normLOfour{\uu}\norm{b}_{H^{1/2,2}(\partial \Omega)}
      +C\norm{b}_{H^{3/2,2}(\partial \Omega)}+C_\varepsilon
          \normLO{\uu}\normLO{\nabla \phizt}.
\]
\end{lemma}
\begin{proof}
Define the continuous bilinear form $b_\phi:\HOzero\times\HOzero\to\rz$ by
$$
b_\phi[\varphi,\psi] = \intO   \Eeu \cdot\nabla \varphi \, \psi \dx  
	+ \intO  \nabla \varphi \cdot \nabla \psi \dx, \qquad
\varphi,\psi\in H_0^{1,2}(\Omega).
$$
The first integral is anti-symmetric because $\Eeu$ is divergence free, 
and $\psi$ vanishes on the boundary, so the bilinear form $b_\phi[\cdot,\cdot]$ is
coercive. 

Now, $\phizt$ solves Eq. \eqref{phi_tilde}, iff 
\[
\phizt \in \HOzero : \qquad b_\phi[\phizt,\psi] = b_\phi[\phi_D,\psi] 
      - \intO \he(\phi)\nabla \mep*T \cdot \nabla \psi \dx , 
        \quad \forall \psi \in \HOzero.
\]
Thus the existence of a unique solution $\phit$ of~\eqref{phi_tilde} follows 
from Lax-Milgram's Lemma. 

To show the $\Hone$-estimate, we test Eq.\ \eqref{phi_tilde} with $\psi=\phizt$, and use the fact
that $\phid$ is harmonic, to obtain
\begin{align*}
&\intO \Eeu \cdot  \nabla \phizt  \,\phizt \dx+\intO  \Eeu  \cdot \nabla \phid \phizt \dx  
+\intO | \nabla \phizt|^2 \dx+\intO \he(\phi)\nabla \mep*T \cdot \nabla \phizt \dx=0.
\end{align*} 
Making use of the anti-symmetry of the first integral again and employing the Cauchy-Schwarz inequality we arrive at
\begin{align*}
\normLO{\nabla \phizt}^2 \leq  \normLOinf{\he(\phi)}\normL{\mep * \nabla T} \normLO{\nabla\phizt} 
+\normLOfour{ \Eeu}\normLO{\nabla \phid}\normLOfour{\phizt}.
\end{align*}
Employing Poincar\'e's inequality and the Sobolev embedding for the last summand, 
noting that $\normLO{\nabla \phid}\leq \norm{b}_{H^{1/2,2}(\Omega)}$, 
Lemma~\ref{Lemma_convolution_mollifier} and the bound 
$\normLOinf{\he(\phi)}\leq\frac{1}{4}$ yield the first assertion. 

The last assertion follows from elliptic regularity as follows: 
Since $\nabla \mep*T\eqqcolon \boldsymbol{s}$ belongs to $\HO\cap\LOinf$ and $\Eeu \in \LlOinf$ and recalling that $\phid$ is harmonic we notice that 
Eq. \eqref{phi_tilde} implies that $\phizt$ is the weak solution to
\begin{align}\label{phi_proof}
-\Delta  \phizt = \nabla \cdot (\he(\phi)\nabla \mep*T) - \Eeu \cdot \nabla (\phizt+\phid) 
\quad\text{in $\Omega$}, \qquad
\phizt = 0\quad\text{ on $\partial\Omega$}, 
\end{align} 
with right-hand side satisfying
\begin{align*}
\nabla \cdot (\he(\phi)\nabla \mep*T) - \Eeu \cdot\nabla (\phizt+\phid)  
&= \nabla \he(\phi)\cdot\boldsymbol{s}+\he(\phi)\nabla \cdot\boldsymbol{s}- \Eeu\cdot \nabla (\phizt+\phid)\\
&=\underbrace{\he'(\phi)}_{\in \Linf}\underbrace{\nabla \phi}_{\in \Lltwo} \cdot\underbrace{ \boldsymbol{s}}_{\in \Llinf}+\underbrace{\he(\phi)}_{\in \Linf}\underbrace{\nabla \cdot\boldsymbol{s}}_{\in \Ltwo} -  \underbrace{\Eeu}_{\in \Llinf}\cdot \underbrace{\nabla (\phizt+\phid)}_{\in \Lltwo}.
\end{align*}
Thus the right hand side is in $\LOtwo$. So by elliptic regularity, 
$\phizt \in \HOtwozero$ and since $\phid \in \HOtwo$ one concludes 
that $\phit \in \HOtwo$ and also the estimate for  $\normHOtwo{ \phit}$ holds.
\end{proof}

\bigskip
For the remainder it is important to know that $\phit$ is bounded pointwise
from below and above. 
\begin{lemma}[Pointwise bounds on $\phie$]\label{L:pointwise}
The solution $\phit$ of Eq.\ \eqref{phi_tilde} fulfills $0\leq \phit\leq 1$
a.e.
\end{lemma}

\begin{proof}
The bounds on $\phit$ can be shown as in \citep{Baensch2018}. As a test function 
we choose $\varphi=\min{\{\phit_0+\phid, 0\}}$ which is in $\HOzero$ since 
$\phit_0+\phid \ge 0$ at the boundary, by construction of $\phit_0$ and $\phid$. 
So we have
\begin{align*}
\underbrace{\int _{\phit<0} \mep*\uu \cdot \nabla \phit  \phit \dx}_{=0,\text{ (anti-symmetric)}} + \int _{\phit<0} |\nabla\phit |^2 \dx 
\quad+\underbrace{\int _{\phit<0} \he(\phit) \nabla \mep*\Te \cdot \nabla \phit \dx}_{=0 \text{ (Def.\ of $\he$)}} =0.
\end{align*}
Consequently using Poincar\'e's inequality we get  
\begin{align*}
0=\int _{\phit<0} |\nabla\phit |^2 \dx =\norm{\nabla \phit}_{L^2(\phit<0)}^2\geq 
C\norm{ \phit}_{L^2(\phit<0)}^2,
\end{align*} 
which shows that $\varphi=\min{\{\phit, 0\}}\equiv0$, thus $\phit \geq 0$ a.e. The upper bound also follows by taking $\varphi=-\min{\{1-(\phit_0+\phid), 0\}}$.
\end{proof}

We notice that  Eq. \eqref{phi_tilde} implies
$\intO  ( \phit\, \Eeu+\jjp) \cdot \nabla \psi \dx=0 \, \text{ for all } \, \psi \in \HOzero$
with the flux $\jjp$ given by Eq. \eqref{flux_tilde}.
Furthermore, since  
$\nabla (\psi^2)=2\psi\nabla\psi\in L^{3/2}(\Omega)$ and $\phit \in \HOtwo$ 
we have $\jjp\in \HhO \hookrightarrow \boldsymbol{L}^6(\Omega)$.  Thus by density
one can take also $\psi^2$ as a test function:

\begin{align}
\intO  ( \phit\, \Eeu+\jjp)  \cdot \nabla (\psi^2)\, \mathrm{d}x=0, \quad \forall \, \psi \in \HOzero 
\label{j_phisquared}
\end{align}

\bigskip
Next we show existence of a unique solution to Eq. \eqref{T_tilde}.
\begin{lemma}{} \label{Lemma_Tt}
For any $0<\varepsilon \leq 1/4$ there exists a unique solution $\Tt$ to Eq.\ \eqref{T_tilde} 
and it satisfies
$\normLO{\nabla \Tt}\leq C \normLO{f}$.
Moreover $\Tt \in \HOtwozero$ and 
$\normHOtwo{ \Tt}\leq C_\varepsilon (\normHO{\phit},\normLO{\uu},\normLO{T}, \normLO{f})$.
\end{lemma}

\begin{proof}
The existence and uniqueness of $\Tt$ follows again from Lax-Milgram theorem. 
Indeed, it is easy to see that the bilinear form $b_T[\cdot,\cdot]$ is bounded in $H_0^1(\Omega)\times H_0^1(\Omega)$. Next, for $\varphi \in H_0^1(\Omega)$, we have
\[
b_T[ \varphi, \varphi ]= \intO \mep*k(\phi) \nabla \varphi\cdot \nabla \varphi \dx 
   +\intO \Eeu \cdot \nabla (\etat \varphi)   \varphi \dx 
- \intO \varphi \jjp \cdot \nabla \varphi \dx.
\]
By partial integration on the second integral and the fact that
\begin{align*}
\intO  \etat \varphi\,  \Eeu \cdot \nabla \varphi \dx =\intO   (1+\phit) \varphi\, \Eeu\cdot \nabla   \varphi \dx 
=\underbrace{\intO \varphi \,\Eeu \cdot \nabla   \varphi \dx}_{=0 \text{ (anti-symmetric)}} +\intO  \phit\, \varphi \,\Eeu \cdot \nabla   \Tt \dx,
\end{align*} 
we get
\[
b_T[ \varphi, \varphi ] = \intO \mep*k(\phi) |\nabla \varphi|^2\dx - \underbrace{\frac{1}{2}\intO ( \phit\, \Eeu +\jjp) \cdot \nabla (\varphi^2)\dx}_{=0 \text{ (Eq.\ \eqref{j_phisquared})}}= \intO \keffmep |\nabla \varphi|^2\dx,
\]
which shows the coercivity of the bilinear form. 
Since $f \in L^2(\Omega)$, there exists a unique solution $\Tt$ to Eq.\ \eqref{T_tilde}, and moreover
\begin{align*}
\normHO{\Tt}^2 \Cle b_T[\Tt,\Tt] \Cle \intO f \Tt \dx \le \normHO{\Tt} \normLO{f},
\end{align*}
which yields the first asserted estimate.

To show that $\Tt \in \HOtwozero$ we note that $\Tt$ is a weak solution to
\begin{align}\label{reg_T}
-\Div(\keffmep\nabla \Tt)= \underbrace{f-\Div\big((\Tt\jjp)- \etat \Tt\,\Eeu\big)}_{=\tilde{f}}
\quad\text{in }\Omega,\qquad
\Tt = 0 \quad\text{on }\partial\Omega. 
\end{align}
Let us now observe that
\begin{align*}
\tilde{f}=\underbrace{f}_{\in \Ltwo}-\underbrace{\underbrace{\nabla\Tt}_{\in \Lltwo}\cdot \underbrace{\jjp}_{\in \Llsix}}_{\in L^{3/2}} -\underbrace{\underbrace{\Tt}_{\in \Lsix}\underbrace{\nabla\cdot\jjp}_{\in \Ltwo}}_{\in L^{3/2}}- \underbrace{\underbrace{\nabla \etat}_{\in \Lltwo} \underbrace{\Tt}_{\in \Lsix}\cdot\underbrace{\Eeu}_{\in \Llinf} }_{\in L^{3/2}} -\underbrace{ \underbrace{\etat}_{\in \Lsix} \underbrace{\nabla  \Tt}_{\in \Lltwo}\cdot\underbrace{\Eeu}_{\in \Llinf}}_{\in L^{3/2}}.
\end{align*}
Therefore, $\tilde{f}\in L^{3/2}(\Omega)$ and, consequently, 
$\Tt\in H^{2,3/2}(\Omega)$ 
by Lemma \ref{Lemma:regularity} and then
\bigskip

$\nabla \Tt \in\begin{cases}
L^6(\Omega) & \text{ for } d=2,\\
L^3(\Omega) & \text{ for } d=3,
\end{cases}$  \qquad and \hfil
$\Tt \in\begin{cases}
L^\infty(\Omega)  & \text{ for } d=2,\\
L^p(\Omega) \text{ with } 1\leq p<\infty & \text{ for } d=3.
\end{cases}$ 
\bigskip

Choosing $p=7$, we get
\begin{align*}
\tilde{f}=\underbrace{f}_{\in \Ltwo}-\underbrace{\underbrace{\nabla\Tt}_{\in L^3}\cdot \underbrace{\jjp}_{\in \Lsix}}_{\in \Ltwo} -\underbrace{\underbrace{\Tt}_{\in L^7}\cdot\underbrace{\nabla\jjp}_{\in \Ltwo}}_{\in L^{14/9}}-\underbrace{ \underbrace{\nabla \etat}_{\in \Lltwo} \underbrace{\Tt}_{\in L^7}\cdot\underbrace{\Eeu}_{\in \Llinf} }_{\in L^{14/9}} -\underbrace{\underbrace{\etat}_{\in \Lsix} \underbrace{\nabla  \Tt}_{\in \boldsymbol{L}^3}\cdot\underbrace{\Eeu}_{\in \Llinf} }_{\in \Ltwo},
\end{align*}
which in turn yields $\tilde{f}\in L^{14/9}(\Omega)$ and, consequently, 
$\Tt\in H^{2,14/9}(\Omega)$ due to Lemma \ref{Lemma:regularity}.
Then one concludes also for $d=3$ that $\Tt \in \Linf(\Omega)$, 
whence $\tilde{f}\in \LOtwo$ and thus $\Tt \in \HOtwozero$, using elliptic regularity
again. The $\Htwo$-estimate also follows directly as part of the
elliptic regularity and the uniqueness of the solution $\Tt$.
\end{proof}

It stands out that Eq.~\eqref{u_tilde} is a vector-valued equation but it has the 
same structure as \eqref{T_tilde}.

\begin{lemma}{} \label{Lemma_ut}
There is a unique solution $\uut$ to Eq.\ \eqref{u_tilde} and 
$\normLO{\nabla \uut}\leq C \big(\normLO{\gb}+\normLO{T}\big)$.
Moreover $\uut \in \Vv \cap \boldsymbol{H}^{2,2}(\Omega)$ and 
$\normHOtwo{ \uut} \leq C_\varepsilon (\normHO{\phit},\normLO{\uu},\normLO{T}, \normLO{\gb})$.
\end{lemma}

\begin{proof}
To see the existence of a unique solution and the estimate, one can proceed exactly as in the first part of the proof of Lemma~\ref{Lemma_Tt}.

To show that $\uut \in \Vv \cap \boldsymbol{H}^{2,2}(\Omega)$, we refer to Thm.~5.2.3 
from  \citep{Abels2008}. There, the stationary Stokes equation with variable
viscosity was considered, 
and it was shown that there exists a unique solution in $\boldsymbol{H}^{2,2}(\Omega)$ with $p \in \HO$ fulfilling the estimate as stated in Lemma~\ref{Lemma_ut}. 
In our case, moving the nonlinear term to the right-hand side and using the same boot-strapping argument as above in the proof of Lemma~\ref{Lemma_Tt} for $\Tt$, the stated regularity $\uut \in \Vv \cap \boldsymbol{H}^{2,2}(\Omega)$ and the $\Htwo$-estimate follow.
\end{proof}

\bigskip
The next step is to show that there is a solution to Eq. \eqref{P:regularized}. To
this end the Leray-Schauder fixed point theorem is used.

\begin{theorem}{Leray-Schauder (\citep[Thm.\ 11.3]{Gilbarg2001})}
\label{Schauder}
Let $G$ be a compact mapping of a Banach space $\mathfrak{B}$ into itself, and suppose there exists a constant $M$ such that 
\begin{align}
\|x\|_\mathfrak{B}<M 
\end{align}
for all $ x\in \mathfrak{B} $ and $\sigma \in [0,1]$ satisfying $x=\sigma Gx$. Then $G$ has a fixed point.\end{theorem}

The rest of this section is devoted to proving that the assumptions of 
the Leray-Schauder theorem hold for the mapping $G$ defined in Eq. \eqref{P:regularized}.
\begin{lemma}\label{L:compact}
Let $0<\varepsilon \leq 1/4$.
The mapping $G : \HOzero \times \HOzero \times \Vv \rightarrow \HOzero \times \HOzero \times \Vv$ defined in Eq. \eqref{P:regularized} is well defined and compact.
\label{G_continuous_bounded}
\end{lemma}

\begin{proof}
Fix an $0<\varepsilon \leq 1/4$.
$G$ is well defined thanks to Lemmas \ref{Lemma_phit}, \ref{Lemma_Tt} and \ref{Lemma_ut}.
To show compactness of the mapping,
we proceed in two steps. We first show that $G$ is continuous from 
$\mathfrak{B}^1:=\HOzero \times \HOzero \times \Vv $ into $ \mathfrak{B}^1$. Secondly, 
we will show that $G$  maps bounded sets from $\mathfrak{B}^1$  into bounded sets 
of $\mathfrak{B}^2 := \HOtwo \times \HOtwo \times \HhOtwo \cap \mathfrak{B}^1$. 

To prove the continuity of $G$, let $\{ (\phi_k,T_k, \uu_k)\}_k$ be a sequence in $\mathfrak{B}^1$ such that 
$(\phi_k,T_k, \uu_k) \to  (\phi,T,\uu)$ in $\mathfrak{B}^1 $, i.e., $\phi_k  \to \phi$, $T_k \to T$ in $\HO$ and $\uu_k \to \uu$ in $\Vv$.
Let  $(\phit_{0,k},\Tt_k,\uut_k)=G(\phi_{0,k},T_k,\uu_k)$ and $(\phit_{0},\Tt,\uut)=G(\phi_{0},T,\uu)$.
Subtracting Eq. \eqref{phi_tilde} for $\phit_k$ and $\phit$ and noting that $\phit_{D,k}=\phit_{D}$, we obtain 
\begin{align} \label{phidif}
\begin{split}
\intO \nabla (\phit_{0,k}-\phit_{0})\cdot \nabla \psi \dx={}&\intO \phit_k \,\mep * ( \uu_k-  \uu) \cdot \nabla \psi \dx
+ \intO (\phit_{0,k}-\phit_{0})\mep * \uu  \cdot \nabla   \psi \dx \\
&-  \intO \big( (\he(\phi_k)-\he(\phi)) \nabla \mep*T_k+\he(\phi) \nabla (\mep*(T_k-T) \big) \cdot \nabla \psi \dx.
\end{split}
\end{align}
Taking $\psi=\phit_{0,k}-\phit_{0}$, noting that the second integral on the right hand side is anti-symmetric and using H\"older's inequality, we obtain
\begin{align*}
\normLO{\nabla(\phit_{0,k}-\phit_{0})} \Cle{}& \normLOthree{\phit_k}\normLOsix{\mep*( \uu_k- \uu)}\\
& + \normLO{\he(\phi_k)-\he(\phi)}\normLOinf{\nabla \mep*T_k}
 + \normLOinf{\nabla \mep}\normLO{T_k-T}.
\end{align*}
Lemma~\ref{Lemma_phit} implies that $\{\normLOthree{\phit_k}\}_{k\in\nz}$ is bounded, 
$|\he(\phi_k)-\he(\phi)|\leq |\phi_k-\phi|$ due to the Lipschitz continuity of $\he$. 
Using Lemma~\ref{Lemma_convolution_mollifier} we arrive at
\begin{align*}
\normLO{\nabla (\phit_{0,k}-\phit_{0})}\Cle {}& \normLOinf{\nabla \mep}  \normLOsix{\uu_k-\uu} \\
  &+ \normLO{\phi_k-\phi} \normLOinf{\nabla \mep*T_k}+ C \normLOinf{\nabla \mep} \normLO{T_k-T}.
\end{align*}
Now, $\normLOinf{\nabla \mep*T_k }$ is bounded due to Lemma~ \ref{Lemma_convolution_mollifier}(\ref{Lemma_convolution_mollifier.v}), whence 
$\normHO{\phit_{k}-\phit} \to 0$ as $k \to \infty$.

The next step is to show that
$\normLO{\nabla(\Tt_k-\Tt)}\rightarrow 0 $ as $k \to \infty$.
We will first show that $\normLOthree{\jj_{p,k}-\jj_{p}} \rightarrow 0 $, which will 
be a consequence of the fact that $\normHOtwo{\phit_k-\phit}\rightarrow 0 $ as 
$k \to\infty$.

In order to show that $\normHOtwo{\phit_{0,k}-\phit_{0}}\to 0$ as $k\to\infty$, we look at the strong form of Eq.~\eqref{phidif}:
\begin{align*}
-\Delta (\phit_{0,k}-\phit_{0})={}&\Div\big( - \phit_k\, \mep *( \uu_k- \uu)  - (\phit_{0,k}-\phit_{0})\mep* \uu \\\
&+(\he(\phi_k)-\he(\phi)) \nabla \mep*T_k+\he(\phi)\nabla \mep *(T_k-T) \big).
\end{align*}
From Lemma \ref{Lemma:regularity} we have
\begin{align*}
\normHOtwo{\phit_{0,k}-\phit_{0}} \Cle {}&
 \normLO{\Div\big(  \phit_k\, \mep*( \uu_k- \uu) + (\phit_{0,k}-\phit_{0})\mep*\uu \big)} \\
&+\normLO{\Div\big((\he(\phi_k)-\he(\phi))\nabla \mep*T_k+\he(\phi)\nabla \mep* (T_k-T)\big)}\\
\le{}& \underbrace{\normLO{ \phit_k\Div\big(\mep*(\uu_k- \uu)\big)  }}_{=\RM{1}} 
+\underbrace{\normLO{ \nabla \phit_k\cdot\mep*( \uu_k- \uu) }}_{=\RM{2}}\\
&+\underbrace{\normLO{(\phit_{0,k}-\phit_{0})\Div\big(\mep*\uu \big)}}_{=\RM{3}}
+\underbrace{\normLO{ \nabla(\phit_{0,k}-\phit_{0})\cdot\mep* \uu) }}_{=\RM{4}}\\
&+\underbrace{\normLO{\nabla (\he(\phi_k)-\he(\phi))\cdot \nabla \mep*T_k }}_{=\RM{5}}+\underbrace{\normLO{(\he(\phi_k)-\he(\phi))\Delta \mep*T_k }}_{=\RM{6}}\\
&+\underbrace{\normLO{\nabla \he(\phi) \cdot \nabla \mep*(T_k-T)}}_{=\RM{7}}
 +\underbrace{\normLO{ \he(\phi)\Delta \mep* (T_k-T)  }}_{=\RM{8}}.
\end{align*}
Observe that $\RM{1}=\RM{3}=0$ because  $\mep* \uu_k$ and $\mep* \uu$ are divergence free, furthermore $\RM{2},\RM{4}, \RM{6}$ and $\RM{8}$ go to  zero as $k \to \infty$, due to the  Lipschitz continuity of $\he$ and Lemma~\ref{Lemma_convolution_mollifier} as seen above. 

Since $\he'(\phi) \le 1$ is bounded for every $\varepsilon>0$, we have that 
\begin{align*}
\RM{7}\leq  \normLO{ \he'(\phi)\nabla \phi }\normLOinf{\nabla \mep*(T_k-T) }  
\leq C_\varepsilon \normLO{ \he'(\phi)\nabla \phi }\normLO{T_k-T  },
\end{align*} 
so that $\RM{7} \rightarrow 0$ as $k\to \infty$, where again we used Lemma~\ref{Lemma_convolution_mollifier}.(\ref{Lemma_convolution_mollifier.v}).

We can bound $\RM{5}$ as follows:
\begin{align*}
\RM{5}&\leq\normLOinf{ \nabla \mep*T_k}\bigg(\underbrace{\normLO{(\he'(\phi_k)-\he'(\phi))\nabla \phi}}_{=\RM{5.\RMn{1}}_k}
+\underbrace{\normLOinf{\he'(\phi_k)}}_{\leq 1}\underbrace{\normLO{\nabla (\phit_k-\phit)}}_{\xrightarrow[k\to\infty]{} 0}\bigg).
\end{align*}
In order to show the convergence of $\RM{5.\RMn{1}}_k$, we rewrite
\begin{align*}
\RM{5.\RMn{1}}_k^2&=\intO|\he'(\phi_k)-\he'(\phi)|^2|\nabla\phi|^2=\intO\big|\he'(\phi_k)|\nabla\phi|-\he'(\phi)|\nabla\phi|\big|^2
\end{align*}
and notice that $|\he'(\phi_k)| |\nabla\phi|\leq |\nabla\phi| \in \LOtwo$ due to the bound of $\he'$.

We now prove that $\RM{5.\RMn{1}}_k^2 \to 0$ by contradiction.
If this were not true, there would exist a subsequence such that $\RM{5.\RMn{1}}_{k_j}^2 > \delta$ for some $\delta > 0$. Since $\phi_{k_j}\strong \phi$ in  $\LOtwo$, there exists 
a further subsequence, which we still call $\phi_{k_j}^n$ that converges to $\phi$ 
almost everywhere. Since $\he'$ is continuous we know that 
$\he'(\phi_{k_j}^n)|\nabla \phi|\strong \he'(\phi)|\nabla \phi|$ a.e.. So we can apply Lebesgue's dominated convergence theorem to the subsequence to conclude that $\intO|\he'(\phi_{k_j}^n)-\he'(\phi)|^2|\nabla\phi|^2\xrightarrow[]{}  0$, which is a contradiction.

Summarizing, we have $\normHOtwo{\phit_{0,k}-\phit_0}\rightarrow 0$ as $k\to\infty$, 
and we will show that 
$\normLOthree{\jj_{p,k}-\jj_{p}} \rightarrow 0$. Observe that
\begin{align*}
|\jj_{p,k}-\jj_{p}|=|\nabla (\phit_k-\phit)+(\he(\phi_k)-\he(\phi))\nabla\mep*T_k 
+\he(\phi)\nabla\mep*(T_k-T)|,
\end{align*}
therefore
\begin{align*}
 \normLOthree{\jj_{p,k}-\jj_{p}} \leq{} \normHOtwo{\phit_{0,k}-\phit_{0}}+\normLOsix{\phi_k-\phi}\normLOsix{\nabla\mep*T_k}
+\frac{1}{4}\normLOthree{\nabla\mep*(T_k-T)},
\end{align*}
which goes to zero due to Sobolev embeddings and Lemma~\ref{Lemma_convolution_mollifier}.(\ref{Lemma_convolution_mollifier.v}).

Now subtracting Eq.\ \eqref{T_tilde} for $\phit_k, \Tt_k $ and $\phit, \Tt$, rewriting $\nabla(\etat\Tt)\varphi = \nabla(\etat\Tt\varphi) - \etat\Tt\nabla\varphi$ and using that $\nabla \cdot \jj_p = \nabla\cdot\jj_{p,k} = 0$ we get
\begin{align*}
\intO\mep* k(\phi_k)&\nabla (\Tt_k-\Tt) \cdot \nabla \varphi \dx \\
={}&  \intO (\Tt_k-\Tt)\jj_{p,k}\cdot \nabla \varphi \dx+  \intO (\Tt_k-\Tt)\mep*\uu_k \etat_k\cdot \nabla \varphi \dx \\
& -\intO  \mep*(k(\phi_k)-k(\phi))\nabla \Tt \cdot \nabla \varphi \dx +\intO \Tt (\jj_{p,k}-\jj_{p}) \cdot \nabla \varphi \dx\\
&+ \intO \Tt\,\mep*(\uu_k-\uu) \etat_k\cdot \nabla \varphi \dx+ \intO (\etat_k-\etat)\Tt\,\mep*\uu\cdot \nabla \varphi \dx.
\end{align*}
Taking $\varphi=\Tt_k-\Tt$ and applying Eq.~\eqref{j_phisquared} we conclude that 
the first two integrals vanish. Noting that $\Tt \in \HOtwozero$ and making use of 
$\normLOthree{\mep*(k(\phi_k)-k(\phi))} \leq C_\varepsilon \normLOthree{k(\phi_k)-k(\phi)}$
and then using the  Lipschitz continuity of $k$ and 
$\eta=1+\phi$ in the last summand yields
\begin{align*}
k_0 \normLO{\nabla(\Tt_k-\Tt)} \Cle {}& \normLsix{\nabla \Tt} 
\normLOthree{\phi_k-\phi} 
+\normLOsix{\Tt}\normLOthree{\jj_{p,k}-\jj_{p}}\\
&+ \normLOthree{\Tt\etat_k}\normLOsix{\uu_k-\uu} 
+ \normLOthree{\phit_k-\phit}\normLOsix{\Tt}\normLOinf{ \mep*\boldsymbol{E}( \uu)} \\
&\xrightarrow[k\to\infty]{} 0.
\end{align*}
A similar computation for Eq.\ \eqref{u_tilde} yields $\normLO{\nabla(\uut_k-\uut)}\strong 0$ as $k\to\infty$.

Having established the continuity of $G$ from $\mathfrak{B}^1$ into $\mathfrak{B}^1$ we now show that $G$ maps bounded sets in $\mathfrak{B}^1$ into  bounded sets of 
$\mathfrak{B}^2$. 
For $\phi \in \HO$, $ T \in \HOzero$, $\uu \in \Vv$, we already know that $\phit \in \HOtwo$ is a weak solution to $-\Delta \phit =\Div (\he(\phi)\nabla \mep*T)-\Eeu\cdot \nabla \phit$ in $\Omega$, $\phit = b$ on $\partial\Omega$. 
Furthermore, we know from the proof of Lemma~\ref{Lemma_phit} that the right-hand side $\Div (\he(\phi)\nabla \mep*T) - \Eeu\cdot \nabla \phit$ belongs to $\LOtwo$, so by elliptic regularity we conclude that 
$\normHOtwo{\phizt}\leq C_\varepsilon \normLO{\Div (\he(\phi)\nabla \mep*T) - \Eeu\cdot \nabla \phit} \Cle C_\varepsilon \| (\phi,T,\uu)\|_{\mathfrak{B}^1}$. We similarly have that $\Tt \in \HOtwo$ is a weak solution to Eq.~\eqref{reg_T} and we know from the proof of Lemma~\ref{Lemma_Tt}, that $\tilde{f} \in \LOtwo$ and also  $\normHOtwo{\Tt}\leq C_\varepsilon \normLO{\tilde{f}} \Cle C_\varepsilon \| (\phi,T,\uu)\|_{\mathfrak{B}^1}$. The bound for $\normHOtwo{\uut}$ was stated in Lemma~\ref{Lemma_ut}. 
\end{proof}

We are now ready to state and prove the main result of this section.

\begin{theorem}
For each $0<\varepsilon\leq 1/4$, there exists a solution $(\phie,\Te, \uue)$ to
\begin{align}
\begin{split}
\phie \in \HO: &\intO \mep*\uue \cdot\nabla \phie \psi \dx +\intO \nabla \phie \cdot \nabla \psi \dx  \\[4pt]
&+ \intO \he(\phie) \nabla \mep*\Te \cdot \nabla \psi \dx=0, \quad \forall \psi \in C^\infty_0(\Omega),
\end{split}\label{phi_epsilon}\\
\begin{split}
\Te \in \HOzero: &\intO \mep*k(\phi^\varepsilon) \nabla \Te \cdot \nabla \varphi \dx + \intO \mep*\uue \cdot\nabla (\etae \Te) \varphi \dx \\[4pt]
&- \intO \Te \jje \cdot \nabla \varphi \dx=\intO f \varphi \dx, \quad \forall \varphi \in C^\infty_0(\Omega),
\end{split}\label{T_epsilon}\\
\begin{split}
\uue \in \Vv: &\intO \frac{\mep*\mu(\phie)}{2} 
D(\uue) : D(\vv) \dx  +\intO \mep* \uue \cdot \nabla (\rhoe\uue) \cdot\vv \dx \\[4pt]
&- \intO \uue \otimes \jje: \nabla \vv \dx +\intO \beta T^\varepsilon \ee_g \cdot \vv \dx  =\intO \gb \cdot \vv \dx,  \quad \forall \vv \in  \boldsymbol{C}^\infty_0(\Omega) \text{ with } \Div \vv =0
\end{split}\label{u_epsilon}
\raisetag{3\normalbaselineskip}
\end{align}
with $\phie=\phie_0+\phi_D$, $\phie_0 \in \HOzero$ and $\jje=-\nabla \phie -\he(\phie) \nabla \mep*\Te$.
Moreover, the following pointwise bounds hold for $\phie$:
$$
 0\leq \phie \leq 1 \quad\text{a.e.}
$$
\end{theorem}

\begin{proof}
The aimed solution is a fixed point of the mapping $G$ defined at the beginning of this section. From Lemma~\ref{G_continuous_bounded} this mapping is compact.
We now notice that if $(\phiz,T,\uu)=\sigma G(\phiz,T,\uu)$ for some 
$\sigma \in [0,1]$, then, 
owing to Lemmas~\ref{Lemma_phit}, \ref{Lemma_Tt} and~\ref{Lemma_ut}
\begin{align*}
\normHO{\phiz} &\le C_\varepsilon \big( \normHO{T} + \normHO{\uu} + \norm{b}_{H^{1/2,2}(\partial \Omega)} \big), \\
\normHO{T} &\le C \normLO{f} , \\
\normHO{\uu} &\le C \big( \normLO{g} + \normHO{T} \big),
\end{align*}
with constants $C, C_\varepsilon$ that are independent of $\phiz$, $T$, $\uu$ and $\sigma$.
Therefore, we have that
\begin{align*}
\|{(\phiz,T,\uu)}\|_{\mathfrak{B}^1}
= \|{(\phiz,T,\uu)}\|_{H^{1,2}(\Omega)\times H^{1,2}(\Omega)\times \Vv}
\leq \tilde C_\varepsilon,
\end{align*}
where $\tilde C_\varepsilon$ depends on the data and $\varepsilon$, but not on  $\phiz$, $T$, $\uu$ and $\sigma$.
Hence, all the assumptions of the Leray-Schauder theorem hold for $G$, and the assertion thus follows.

The pointwise bounds on $\phie$ follows in the same way as in the proof of 
Lemma \ref{L:pointwise}.
\end{proof}

\subsection{Convergence to a solution}

In this section we will prove the existence of a weak solution to the original problem, stated in Definition~\ref{Def_weak}. We will show that there exists a sequence $\varepsilon_n \to 0$ such that the solutions of the mollified problems converge to a solution of the original problem.

From Lemmas~\ref{Lemma_Tt} and~\ref{Lemma_ut}, we have that  $\normHone{\Te}$ and $\normV{\uue}$  are bounded independently of $\varepsilon$. In order to conclude weak $\Hone$ and strong $\Ltwo$ convergence we only need to check if this also holds for $\normHone{\phie}$. 
We rewrite $\nabla \mep*\Te= \mep*\nabla \Te$ which we can do by Lemma~\ref{Lemma_convolution_mollifier} because $\Te \in \HO$. Since $ \COinfzero$ is dense in $\HOzero$, we can test~Eq.\ \eqref{phi_epsilon} with $\psi=\phie_0$. This results in 
\begin{align*}
\intO \nabla\phie_0 \cdot \nabla\phie_0 \dx = -\intO \mep*\uue\cdot \nabla\phi_D\phie_0\dx
- \intO \nabla\phi_D \nabla\phie_0 \dx- \intO \he(\phie)\, \mep*\nabla \Te\cdot \nabla\phie_0 \dx,
\end{align*}
which implies that
\begin{align*}
\normLO{\nabla\phie_0} \Cle \underbrace{\normLOinf{\he(\phie)}}_{\leq\frac{1}{4}}\underbrace{\normLO{\mep*\nabla \Te }}_{\stackrel{\text{Lemma\ } \ref{Lemma_convolution_mollifier}}{\leq C}\normLO{\nabla \Te}}
+\normHO{\uue}\norm{b}_{H^{1/2,2}(\partial \Omega)} + \norm{b}_{H^{1/2,2}(\partial \Omega)} .
\end{align*}
Therefore, from Lemmas~\ref{Lemma_Tt} and~\ref{Lemma_ut}
\begin{align*}
\normLO{\nabla\phie} \Cle \normLO{f}+\normLO{\gb}\norm{b}_{H^{1/2,2}(\partial \Omega)}+\norm{b}_{H^{1/2,2}(\partial \Omega)}.
\end{align*}
We have thus proved that $\normHO{\phie}$, $\normHO{\Te}$ and  $\normV{\uue}$ are uniformly bounded and thus there exists a sequence $\varepsilon_n\strong 0$ and $(\phi, T, \uu) \in H^{1,2}(\Omega)\times H^{1,2}(\Omega)\times \Vv$ such that
\begin{align}\label{strong}
&\phi^{\varepsilon_n}\strong \phi \text{ in } \LOtwo \text{ and }\LOthree, \quad  &T^{\varepsilon_n}\strong T \text{ in } \LOtwo, \quad \quad \quad \quad&\uu^{\varepsilon_n}\strong \uu \text{ in } \LlOtwo,\\
\label{weak}
&\phi^{\varepsilon_n} \weak \phi \text{ in } \HO, \quad  &T^{\varepsilon_n}\weak T \text{ in } \HO, \quad \quad \quad \quad&\uu^{\varepsilon_n}\weak \uu \text{ in } \Vv.
\end{align}

To see that $(\phi, T, \uu)$ is a solution to the nonlinear problem \eqref{phi_orig}-\eqref{u_orig}, we use the strong $\Ltwo$- and weak $\Hone$-convergence and show that the expressions in~\eqref{phi_epsilon}--\eqref{u_epsilon} converge to the corresponding ones in~\eqref{phi_orig}--\eqref{u_orig} as $\varepsilon\strong0$. To simplify the notation the sequences $\phi^{\varepsilon_n}$, $T^{\varepsilon_n}$ and $\uu^{\varepsilon_n}$ will be denoted by $\phie$, $\Te$ and $\uue$ in the following.

First, for a fixed $\psi \in C^\infty_0(\Omega)$  we subtract the left-hand side of~Eq.\ \eqref{phi_orig} from the left-hand side of~Eq.\ \eqref{phi_epsilon} to obtain:
\begin{align*}
\underbrace{\intO ( \mep  \uue\cdot \nabla \phie -\uu\cdot \nabla \phi ) \psi \dx}_{=\RM{1}} +\intO (\nabla \phie  -\nabla \phi )\cdot \nabla \psi \dx 
+\underbrace{\intO( \he(\phie) \mep* \nabla \Te-h(\phi)\nabla T) \cdot \nabla \psi \dx}_{=\RM{2}}.
\end{align*}
The second integral goes to zero because $\phie\weak\phi$ in $\HO$. We can bound $\RM{1}$ as follows:
\begin{align*}
\RM{1}&\leq \big|\intO( \mep*\uue-\mep*\uu)\cdot \nabla \phie  \psi \dx\big|
+\big|\intO (\mep*\uu-\uu) \cdot \nabla \phie \psi\dx\big|+\big|\intO \big( \uu\cdot\nabla \phie - \uu\cdot\nabla \phi  \big)\psi\dx\big|.
\end{align*}
The first and second terms go to zero due to Lemma~\ref{Lemma_convolution_mollifier}, the fact that $\uue \to \uu$ in $\LOtwo$ and the uniform bound on $\normLO{\nabla \phie}$. The third one converges to zero because of the weak convergence of $\phie$ to $\phi$ in $\HO$.
The remaining term $\RM{2}$ can be rewritten as 
\begin{align*}
\RM{2}&\leq\underbrace{\big|\intO (\he(\phie)-h(\phi))\mep*\nabla \Te\cdot \nabla \psi \dx\big|}_{=\RM{2.\RMn{1}}}\\
&\quad+\big|\intO h(\phi)\nabla (\mep* \Te- \Te)\cdot \nabla \psi \dx\big|+\big|\intO h(\phi)\nabla (\Te- T)\cdot \nabla \psi \dx\big|.
\end{align*}
Here, the second integral goes to zero because of Lemma~\ref{Lemma_convolution_mollifier} and the fact that $\normHO{\Te}$ is uniformly bounded. The third integral goes to zero simply because of the weak convergence  $\Te\weak T$ in $\HO$. After adding $\mp \he(\phi) \mep*\nabla \Te\cdot \nabla \psi$ to the first integral, $\RM{2.\RMn{1}}$ can be estimated as follows:
\begin{align*}
\RM{2.\RMn{1}}\leq{}& \Big(\normLO{\he(\phie)-\he(\phi)} +
\normLO{\he(\phi)-h(\phi)} \Big) \normLO{\mep*\nabla \Te} \normLOinf{\nabla \psi}.
\end{align*}
The first norm goes to zero because of Lipschitz continuity of $\he$ and the strong convergence of $\phie \strong \phi$ in $\LOtwo$. The second one also goes to zero because $|\he(\phi)-h(\phi)| \leq \varepsilon-\varepsilon^2 \to 0$.  
Using Lemma~\ref{Lemma_convolution_mollifier}, $\normLO{\mep*\nabla \Te}$ can be bounded by $\normLO{\nabla \Te}$, which is uniformly bounded. Therefore,~\eqref{phi_orig} holds for any $\psi \in \COinfzero$.

Next, for a fixed $\varphi \in C^\infty_0(\Omega)$ we subtract the left-hand side of~\eqref{T_orig} from the left-hand side of~\eqref{T_epsilon}:
\begin{multline*}
\underbrace{\intO \big( \mep*k(\phi^\varepsilon) \nabla \Te - \keff \nabla T\big) \cdot \nabla \varphi \dx}_{=\RM{1}}
+\underbrace{\intO \big(\mep* \uue \cdot\nabla (\etae \Te) - \uu \cdot\nabla(\eta T) \big) \varphi \dx}_{=\RM{2}}
\\
+\underbrace{\intO  \big(\Te \nabla \phie - T \nabla \phi\big) \cdot \nabla \varphi \dx}_{=\RM{3}} 
+ \underbrace{\intO  \big(\Te  \he(\phie) \mep*\nabla \Te  - T h(\phi)\nabla T\big)\cdot \nabla \varphi \dx}_{=\RM{4}}.
\end{multline*}
$\RM{1}$ can be bounded as follows:
\begin{align*}
\RM{1}&\leq \big|\intO\big( \mep*k(\phi^\varepsilon)-\mep*k(\phi)\big)\nabla \Te \cdot \nabla \varphi \dx\big|\\
&\quad+\big|\intO\big(\mep* k(\phi)-k(\phi)\big)\nabla \Te \cdot \nabla \varphi\dx\big|+\big|\intO  k(\phi) (\nabla \Te-\nabla T )\cdot \nabla \varphi\dx\big|.
\end{align*}
The first and second terms go to zero due to Lemma~\ref{Lemma_convolution_mollifier}, the Lipschitz continuity of $k$ and the uniform bound on $\normLO{\nabla \Te}$.
The third one converges to zero because $\Te$ converges weakly to $T$ in $\HO$.

In $\RM{2}$, we integrate by parts to obtain
\begin{align*}
\RM{2} &=  \intO \big( \mep*\nabla\cdot\uue \etae\Te - \nabla\cdot \uu \eta T \big) \varphi \dx 
+ \intO \big( \mep * \uue \,\etae\, \Te - \uu \,\eta\, T )\cdot \nabla \varphi \dx,
\end{align*}
and the first term vanishes because $\Div \uu = \Div \uue = 0$. 
We now add and subtract mixed terms to bound $\RM{2}$ as follows:
\begin{align*}
|\RM{2}|\leq {}&  \big| \intO \mep * \uue \etae (\Te - T ) \cdot \nabla \varphi \dx \big|
+ \big| \intO \mep * \uue (\etae - \eta) T  \cdot \nabla \varphi  \dx \big| 
\\
&+ \big| \intO \mep * (\uue - \uu)\eta\, T  \cdot \nabla \varphi   \dx \big|
+ \big| \intO  (\mep*\uu - \uu)\eta\, T  \cdot \nabla \varphi \dx \big|.
\end{align*}
All these terms converge to zero as $\varepsilon \to 0$ due to the strong convergence of $\phie$, $\Te$, $\uue$ to $\phi$, $T$, $\uu$ in $\LOtwo$, using the fact that $\etae = 1+\phie$ and $\eta = 1 + \phi$ are bounded in $\LOinf$.

To see that $\RM{3}$ also converges to zero the integral is split up into two parts:
\begin{align*}
\RM{3}\leq \bigg|\intO(\Te-T)\nabla\phie\cdot \nabla \varphi \dx\bigg|+ \bigg|\intO T\nabla(\phi-\phie)\cdot \nabla \varphi \dx\bigg|,
 \end{align*}
and these go to zero simply due to the weak convergence in $\HO$ and strong convergence in $\LOtwo$ of $\phie$ to $\phi$ and of $\Te$ to $T$.

We now split $\RM{4}$ as follows:
 \begin{align*}
\RM{4}&= \underbrace{\intO \Te \he(\phie) \mep*\nabla (\Te-T) \cdot \nabla \varphi \dx}_{=\RM{4.\RMn{1}}}  
+\underbrace{\intO \big(\Te \he(\phie)\mep* \nabla T -T h(\phi)\nabla T \big)\cdot \nabla \varphi\dx}_{=\RM{4.\RMn{2}}},
\end{align*}
and insert a mixed term in $\RM{4.\RMn{1}}$ to arrive at:
\begin{align*}
\RM{4.\RMn{1}}&\leq\underbrace{\bigg|\intO \big(\Te \he(\phie) -T h(\phi)\big)  \mep*\nabla( \Te-T) \nabla \varphi\dx\bigg|}_{=\RM{4.\RMn{1}.\RMn{1}}} 
+\underbrace{\bigg|\intO T h(\phi) \mep* \nabla (\Te-T)\cdot \nabla \varphi \dx\bigg|}_{=\RM{4.\RMn{1}.\RMn{2}}}.
\end{align*}
We rewrite $\RM{4.\RMn{1}.\RMn{1}}$ as
\begin{align*}
\RM{4.\RMn{1}.\RMn{1}}&\leq\bigg|\intO \big(\Te \he(\phie)\nabla \varphi -T h(\phi)\nabla \varphi\big) \cdot \mep*\nabla( \Te-T) \dx\bigg|\\
&\leq \intO \Big[ |\Te \he(\phie) -\Te \he(\phi)|+ |\Te \he(\phi) -\Te h(\phi)|
+ |\Te h(\phi) -T h(\phi)|\Big]|\nabla \varphi \cdot \mep*\nabla( \Te-T) |\dx\\
&\leq \normLOthree{\he(\phie) -\he(\phi)} \normLOsix{\Te} \normLO{ \mep*\nabla( \Te-T)} \normLOinf{ \nabla \varphi} \\
&\quad +\normLOthree{ \he(\phi) -h(\phi)} \normLOsix{\Te} \normLO{ \mep*\nabla( \Te-T)} \normLOinf{ \nabla \varphi} \\
&\quad+\normLO{\Te-T} \normLO{h(\phi)\nabla \varphi \cdot \mep*\nabla( \Te-T)},
\end{align*}
and notice that $\| \Te \|_{\LOsix} \le \normHO{\Te}$, $\normLO{ \mep*\nabla( \Te-T)}$ and $\normLO{h(\phi)\nabla \varphi \cdot \mep*\nabla( \Te-T)}$ are uniformly bounded, due to Lemma~\ref{Lemma_convolution_mollifier} and the fact that $|h(z)| \le 1/4$.
Since $\he$ is Lipschitz continuous with uniform constant one, $\normLOthree{\he(\phie) -\he(\phi)} \le \normLOthree{\phie -\phi} \to 0$ due to~\eqref{strong}. Finally, the fact that $|\he(\phi)-h(\phi)| \leq \varepsilon-\varepsilon^2$ and the strong convergence of $\Te$ to $T$ in $\LOtwo$, imply that $\RM{4.\RMn{1}.\RMn{1}}\to 0$ with $\varepsilon\to 0$.

We now focus on $\RM{4.\RMn{1}.\RMn{2}}$. Notice that
\begin{align*}
\RM{4.\RMn{1}.\RMn{2}}&=\intO \Big[\mep *\big(T h(\phi)\nabla \varphi \big) - T h(\phi)\nabla \varphi\Big] \cdot \nabla (\Te-T) \dx
+\intO T h(\phi)\nabla \varphi\cdot  \nabla (\Te-T)\dx\\
&\leq \normLO{\mep* \big(T h(\phi)\nabla \varphi\big)  - T h(\phi)\nabla \varphi}
\normLO{\nabla (\Te-T)}
+\intO T h(\phi)\nabla \varphi \cdot\nabla (\Te-T) \dx.
\end{align*}
The first term in the last expression converges to zero due to Lemma~\ref{Lemma_convolution_mollifier} and the fact that $\normLO{\nabla (\Te-T)}$ is uniformly bounded. The second term goes to zero because $ T h(\phi)\nabla \varphi \in \Ltwo$ and $\Te \to T$, weakly in $\HO$.

Lastly, we need to look at $\RM{4.\RMn{2}}$ which will also be split up by adding and subtracting terms:
\begin{align*}
\RM{4.\RMn{2}}&\leq |\intO \big(\Te \he(\phie)\mep* \nabla T  -T h(\phi)\nabla T \big)\cdot \nabla \varphi\dx| \\
&\leq \normLOthree{\he(\phie)-\he(\phi)} \normLOsix{\Te} \normLO{\mep*\nabla T}\normLOinf{ \nabla \varphi}\\
&\quad+\normLOthree{\he(\phi)-h(\phi)} \normLOsix{\Te} \normLO{\mep*\nabla T}\normLOinf{ \nabla \varphi}\\
&\quad+\normLO{\Te-T}\normLOinf{h(\phi)}\normLO{\mep*\nabla T} \normLOinf{\nabla \varphi}\\
&\quad+\normLO{\mep*\nabla T-\nabla T}\normLOinf{h(\phi)}\normLO{T}\normLOinf{\nabla \varphi}.
\end{align*}
This last expression converges to zero thanks to Lemma~\ref{Lemma_convolution_mollifier}, the uniform boundedness of $\normHO{\Te}$, the strong convergence of $\Te$ to $T$ in $\LOtwo$ and of $\phie$ to $\phi$ in $\LOthree$ and the properties of $h$, $\he$ mentioned above.

The convergence of the left-hand side of \eqref{u_epsilon} to the left-hand side of~\eqref{u_orig} can be shown analogously.

We have thus arrived at the main result of this article.

\begin{theorem}[Existence of a weak solution]
\label{existence}
Under the assumptions made in Section~\ref{S:NotationAssumptions} there exists a weak solution $(\phi,T,\uu)$ of problem 
\eqref{stationary_strong_phi}--\eqref{stationary_strong_div_u}
  with boundary conditions \eqref{boundary} in the sense of Definition~\ref{Def_weak}.
\end{theorem}
\begin{proof}
The Leray-Schauder theorem (Thm.\ \ref{Schauder}) yields the existence of a solution  $(\phie, \Te, \uue)$ to problem\ \eqref{phi_epsilon}--\eqref{u_epsilon}, for each $\varepsilon  > 0$. The arguments at the beginning of this section show that there is a sequence $\varepsilon_n \searrow 0$ such that $(\phi^{\varepsilon_n}, T^{\varepsilon_n}, \uu^{\varepsilon_n}) \to (\phi, T, \uu)$ as $n\to\infty$. 
The discussion preceding the statement of this theorem shows that $ (\phi, T, \uu)$ is a weak solution to our stationary problem, according to Definition~\ref{Def_weak}.
\end{proof}

\begin{remark}[Non-uniqueness]
The stationary Navier-Stokes equations with constant viscosity
are known to have multiple solutions (if the Reynolds number is not too
small), see \cite[Ch. II \S4]{Temam:84}. Therefore, one cannot expect
uniqueness for the stationary system under consideration.
\end{remark}

\section*{Acknowledgements}
Pedro Morin was partially supported by Agencia Nacional de Promoci\'on Cient\'ifica y Tecnol\'ogica, through grants PICT-2014-2522, PICT-2016-1983, by CONICET through PIP 2015 11220150100661, and by Universidad Nacional del Litoral through grants CAI+D 2016-50420150100022LI. A research stay at Universit\"at Erlangen was partially supported by the Simons Foundation and by the Mathematisches Forschungsinstitut Oberwolfach as well as by the DFG--RTG
2339 \emph{IntComSin}.

%
%

\def\cprime{$'$} \def\cprime{$'$} \def\cprime{$'$} \def\cprime{$'$}
  \def\cprime{$'$} \def\cprime{$'$} \def\cprime{$'$} \def\cprime{$'$}

\bibliography{bibliography}
\addcontentsline{toc}{section}{References}
\bibliographystyle{plain}

\vfill

{\small
\begin{minipage}[t]{0.45\textwidth}
\noindent
Eberhard B\"ansch\\
Applied Mathematics III\\
University Erlangen--N\"urnberg\\
Cauerstr. 11\\
91058 Erlangen\\
Germany\\
{\tt baensch@math.fau.de}
\end{minipage}\hfill
}
{\small
\begin{minipage}[t]{0.45\textwidth}
\noindent
Sara Faghih-Naini\\
Scientific Computing\\
University of Bayreuth\\
Universit\"atsstr. 30\\
95444 Bayreuth\\
Germany\\
{\tt Sara.Faghih-Naini@uni-bayreuth.de}
\end{minipage}\hfill
}\\[1cm]
{\small
\begin{minipage}[t]{0.45\textwidth}
\noindent
Pedro Morin\\
Facultad de Ingenier\'ia Qu\'imica\\
Universidad Nacional del Litoral and CONICET \\
Santiago del Estero 2829 \\
S3000AOM Santa Fe \\
Argentina\\
{\tt pmorin@fiq.unl.edu.ar}
\end{minipage}\hfill
}

\end{document}